\documentclass[12pt,a4paper,reqno]{amsart}

\usepackage[T1]{fontenc}
\usepackage[utf8]{inputenc}
\usepackage{lmodern}
\usepackage[english]{babel}
\usepackage[activate]{pdfcprot}
\usepackage{enumerate}
\usepackage{amssymb}
\usepackage{url}
\usepackage{tikz-cd}

\usepackage{verbatim} 
\usepackage{marginnote}

\usepackage{comment}

\usepackage{datetime}

\usepackage{soul}
\usepackage{amssymb}
\usepackage{amscd}
\usepackage{amsfonts}
\usepackage{mathtools}
\usepackage{enumitem}
\setlist{beginpenalty=100}

\newtheorem{thm}{Theorem}[section]
\newtheorem{lem}[thm]{Lemma}
\newtheorem{cor}[thm]{Corollary}

\newtheorem{conj}[thm]{Conjecture}

\theoremstyle{definition}
\newtheorem{de}[thm]{Definition}

\theoremstyle{remark}
\newtheorem{rem}[thm]{Remark}

\numberwithin{equation}{section}

\newcommand{\Rmnum}[1]{\expandafter\@slowromancap\romannumeral #1@}

\newcommand{\R}{\mathbb{R}}
\newcommand{\N}{\mathbb{N}}
\newcommand{\K}{\mathbb{K}}
\DeclareMathOperator{\SL}{SL}

\newcommand{\Z}{\mathbb{Z}}

\newcommand{\F}{\mathbb{F}}

\newcommand{\til}{\widetilde}

\DeclarePairedDelimiterX\set[1]{\lbrace}{\rbrace}{%

#1
}
\usepackage{xcolor}

\allowdisplaybreaks

\begin{document}

\title{Bounded diagonal orbits in homogeneous spaces over function fields}

\author{Qianlin Huang}
\address{Shanghai Center for Mathematical Sciences, Jiangwan Campus, Fudan University, No.2005 Songhu Road, Shanghai, 200438, China}
\email{qlhuang22@m.fudan.edu.cn}

\author{Ronggang Shi}
\address{Shanghai Center for Mathematical Sciences, Jiangwan Campus, Fudan University, No.2005 Songhu Road, Shanghai, 200438, China}
 \email{ronggang@fudan.edu.cn}
\thanks{NSF Shanghai 22ZR1406200, NSFC 12161141014
and the New Cornerstone Science Foundation.}

\subjclass[2010]{Primary 37B05; Secondary 37C85, 14M17}

\keywords{topological rigidity, homogeneous space,  function field}


\begin{abstract}
This paper is about topological rigidity of diagonal group actions on
the homogeneous space $\SL_4\big(\F(\!(t^{-1})\!)\big)/\SL_4(\F[t])$ where $\F$ is a finite field 
of characteristic $3,5,7$ or $11$. We show that there is a non-closed relatively compact orbit of
the diagonal group. 
\end{abstract}

\maketitle


\section{Introduction}\label{sec;introduction}
The rigidity of subgroup actions on homogeneous spaces has  a long history. One of the main challenges is the following conjecture  due to Cassels-Swinnerton-Dyer \cite{Cassels55} for $n=3$ and due to Margulis \cite{margulis2000} in general.

\begin{conj}
    Let $D$ be the diagonal group of $\SL_n( \R)$ where $n\ge 3$. Then 
any relatively compact $D$-orbit in $\SL_n(\R)/\SL_n(\Z)$ is closed. 
\end{conj}

One of the motivations of this conjecture is that it implies Littlewood conjecture. We refer the readers to Einsiedler-Katok-Lindenstrauss \cite{EKL} for the  progress in this direction.  The aim of this paper is to discuss rigidity of diagonal group actions on  homogeneous spaces over a positive characteristic local field.
Let 
$\K=\F(\!(t^{-1})\!)$ 
be  the field of  formal Laurent series over a finite   field $\F$ with odd characteristic $p$, and let $\F[t]\subset \K$ be the one variable polynomial ring over $\F$.
We endow $\K$ with an ultrametric absolute value $|\cdot|$ so that
$|bt^l|=p^l$ for $b\in \F\setminus\{ 0\}$ and $ l\in \Z$.
It induces a locally compact topology on $\K$ and hence on $\SL_n(\K) $.
For a commutative ring $R$ with identity, the group 
$\SL_n(R)$ consists of $n$-by-$n$ matrices with entries in $R$ and determinant $1$. In this paper $0$ and $1$ denote the zero and one in a ring according to 
the context.
The following is our main result.
\begin{thm}
    \label{thm;main}
    Let $D$ be the diagonal subgroup of $\SL_4(\K)$, where $\K$ is of characteristic $3,5,7$ or $11$. Then there exists $z\in 
   \SL_4(\K)/\SL_4(\F[t]) $ such that $D z$ is a  non-closed  relatively compact  subset. 
\end{thm}

A relatively compact subset of a homogeneous space is said to be bounded. 
Our proof of Theorem \ref{thm;main} is based on the counter examples of $t$-adic Littlewood conjecture constructed  in Adiceam-Nesharim-Lunnon \cite{Erez} and Garrett-Robertson \cite{garrett2025counterexamplesptadiclittlewoodconjecture}. Each counter example corresponds to a 
bounded trajectory of a diagonal cone on 
$$G'/\Gamma'=\SL_2(\K\times \F(t)_{\mathfrak p})/ \SL_2(\F[t, t^{-1}])$$ where $\F(t)_\mathfrak p=\F(\!(t)\!)$ is the completion of 
$\F(t)$ with respect to the absolute value  $|t |_{\mathfrak p}= p^{-1}$ corresponding to the prime ideal $\mathfrak p= t \F[t]$ of $\F[t]$ and $\F[t, t^{-1}]$ embedding diagonally in $\K\times \F(t)_{\mathfrak p}$.

In the case of real numbers, the completion with respect to the usual absolute value is quite different from that with respect to the $p$-adic absolute value, whereas in the positive characteristic case, all the absolute values of degree one of the field  $\F(t)$ are essentially the same.
Using this observation, we show that there is a bounded trajectory of a diagonal cone
 on $$G/\Gamma=\SL_2(\K\times \K)/ \SL_2(\F[t, \sqrt {t+t^2}]).$$ 
 Then we use dynamical arguments on limit points of the cone  to show that there must be a non-closed bounded  orbit 
 of the diagonal group on $G/\Gamma$. We finally obtain a non-compact 
 bounded $D$-orbit in Theorem \ref{thm;main} by pushing the one on
  $G/\Gamma$   into $\SL_4(\K)/\SL_4(\F[t])$ via a naturally defined proper embedding constructed  by restriction of scalars. 

We remark here that  besides topological rigidity,   a measure rigidity of diagonal group action on positive characteristic homogeneous spaces is obtained by Einsiedler-Lindenstrauss-Mohammadi \cite{ELM}.

  \subsection*{Acknowledgments}
The authors would like to thank Jialun Li for the discussions and  Noy Soffer Aranov for drawing our attention to \cite{garrett2025counterexamplesptadiclittlewoodconjecture} and  comments to a preliminary version of the paper.

\section{Dynamical Interpretation of $t$-adic Littlewood}
Let 
$\phi: \F(t)\to \K \times \F(t)_{\mathfrak p} 
$
be the embedding defined by $\phi(\alpha)= (\alpha, \alpha)$ for $\alpha\in \F(t)$. 
The group   $\Gamma'=\SL_2\big(\phi(\F[t, t^{-1}])\big)$ is a discrete subgroup of
$G'$.
For $\alpha\in R$, where $R$ is a commutative ring with identity, 
let 
\begin{align}\label{eq;notation}
u(\alpha)=\begin{pmatrix}
    1 & \alpha \\
    0 & 1
\end{pmatrix} \mbox{ and } a(\alpha)=\begin{pmatrix}
    \alpha & 0\\
    0 & \alpha^{-1}
\end{pmatrix}\mbox{ if }\alpha \mbox{ is invertible}.
\end{align}
For example,  $R$ could be $\K \times \F(t)_\mathfrak p$ or $\K \times \K$, 
in which case 
if  $(\alpha_1,\alpha_2)\in  R$, we write $u(\alpha_1, \alpha_2) $ or 
$a(\alpha_1, \alpha_2)$ in the notation (\ref{eq;notation}).
Let $A_{\infty, \mathfrak p}$ be the diagonal group of $G'
$  and 
let 
\[
A^+_{\infty, \mathfrak p}=
\left\{
a(\alpha_1, \alpha_2): (\alpha_1,\alpha_2)\in \K \times \F(t)_{\mathfrak p},  |\alpha_1|\ge |\alpha_2|_{\mathfrak p}\ge 1  \right\}.
\]
The aim of this section is to prove the following lemma.
\begin{lem}\label{lem;trajectory}
There exists an $\alpha\in \K$ 
transcendental over $\F(t)$ 
such that 
the trajectory $A^+_{\infty, \mathfrak p}u(\alpha, 0)\Gamma'$ 
is bounded in $G'/\Gamma'$.
\end{lem}

The main tools to prove Lemma \ref{lem;trajectory} are \cite{Erez} and \cite{garrett2025counterexamplesptadiclittlewoodconjecture}. To state their results we 
recall some definitions. 
For $\alpha=\sum_{n=l}^\infty b_nt^{-n}\in \K $, the
fractional part of $\alpha$ is
\[
\langle \alpha \rangle = \sum_{n=1}^\infty b_nt^{-n}.
\]
For integers $m$ and $n\ge 1$ we let  $\lfloor m \mod n \rfloor$ to be the smallest nonnegative 
integer $k$ with $m\equiv k \mod n$. 
\begin{de}\label{de;paper}
   For $m\ge 1$, the $m$th level paperfolding sequence $\{f_n\}_{n\ge 1}$ is 
\begin{align}\label{eq;fn}
     f_n=\frac{1}{2}\left\lfloor {k-1 \mod 2^{m+1}} \right\rfloor,
\end{align}
where  $k$ be the odd part of the unique decomposition of $n$ into a product of a power
of $2$ and an odd integer, i.e., $n= 2^l
\cdot k$.
\end{de}

There is a uniquely and naturally defined ring homomorphism from $\Z\to \F_p\subset \F$.
 We will identify $f\in \Z$ with its image in $\F$ according to the context. 

\begin{thm}\cite[Theorem 1.1]{Erez}\label{counterexample}\cite[Theorems 1.5, 1.9]{garrett2025counterexamplesptadiclittlewoodconjecture}
Suppose  the characteristic of $\F$ is $p=3,5,7$ or $11$.
Let $m$ be the maximal positive integer such that $2^{m}|p-1$.
Let  $\alpha=\sum_{n=1}^\infty f_nt^{-n}\in \K $ where  $\{f_n\}_{n\ge 1}$ is the  $m$th paperfolding sequence. 
Then we have 
    \[
    \inf_{0\neq N\in \F[t], k\ge 0} |N|\cdot  |\langle Nt^k\alpha\rangle | >0.
    \]
\end{thm}

\begin{rem}
    In fact, Garrett and Robertson gave a further conjecture \cite[Conjecture 1.7]{garrett2025counterexamplesptadiclittlewoodconjecture}, extending the one given by Adiceam-Nesharim-Lunnon \cite[Conjecture 7.3]{Erez}, that Theorem \ref{counterexample} holds for any odd characteristic finite field $\F$. Our proof of Theorem \ref{thm;main}
    works for any odd characteristic $\F$ for which this conjecture holds.
\end{rem}

\begin{lem}\label{transcendence}
    The element $\alpha\in\K$ in Theorem \ref{counterexample} is transcendental over $\F(t)$ for any finite field $\F$ of odd characteristic.
\end{lem}

To prove the transcendence, we need to introduce the definition of $q$-automatic as well as  Christol's Theorem. The proof can  be found in  \cite[Theorem 12.2.5]{2003Automatic}.

\begin{de}\label{def;auto}
Let \( q \geq 2 \) be a prime power. A sequence \( \{f_n\}_{n \geq 0} \) with values in $\F_q$ is called \textbf{\( q \)-automatic} if there exists a $4$-tuple \( (S, s_0, \mu, \Phi) \) such that:
\begin{itemize}
    \item \( S \) is a finite set,
    \item \( s_0 \in S \) is a selected element,
    \item \(\mu= \{\zeta_0,\zeta_1,\ldots,\zeta_{q-1}  \}  \) consists of maps from $S$ to $S$,
    \item \( \Phi: S \to \F_q \) is a function,
\end{itemize}
and for every \( n \geq 0 \), if \( n = \sum_{k=0}^m d_k q^k \) with \( d_k \in \{0,1,\ldots,q-1  \} \) and $d_m\neq 0$ (except $n=0=0\cdot q^0$), then
\begin{align}\label{eq;auto}
    f_n = \Phi\left(\zeta_{d_m}\circ \zeta_{d_{m-1}}\circ \cdots \circ \zeta_{d_0} (s_0)\right).
\end{align}

\end{de}

\begin{thm}[Christol's Theorem]
Let \( \mathbb{F}_q \) be a finite field, and let \( \{ f_n \}_{n \geq 0} \) be a sequence with values in \( \mathbb{F}_q \). Then the sequence \( \{f_n\} \) is \( q \)-automatic if and only if the formal power series
\[
\sum_{n=0}^\infty f_n t^{-n}
\]
is algebraic over \( \mathbb{F}_q(t) \).
\end{thm}

For example, 
 $\alpha=\sum_{n=0}^\infty t^{-3^n}$ is a root of $x^3-x+t^{-1}=0$.
 It is straight-forward to check that the coefficient sequence $f_n$ of $\alpha$ is $3$-automatic by taking   \( S=\{s_0,s_1,s_2\} \),
\(\mu= \{\zeta_0,\zeta_1,\zeta_2  \}  \) with $\zeta_i(s_j)$ defined according to  the tabular below, and $\Phi(s_0)=\Phi(s_2)=0, \Phi(s_1)=1$.

\begin{center}
\begin{tabular}{|c||c|c|c|}
\hline
$\zeta_i ( s_j )$ & $s_0$ & $s_1$ & $s_2$ \\ 
\hline\hline
$\zeta_0$ & \textcolor{red}{\textbf{$s_0$}} & $s_2$ & $s_2$ \\ 
\hline
$\zeta_1$ & \textcolor{blue}{\textbf{$s_1$}} & $s_2$ & $s_2$ \\ 
\hline
$\zeta_2$ & $s_2$ & $s_2$ & $s_2$ \\ 
\hline
\end{tabular}
\end{center}

\begin{lem}\label{lem;homo}
Let $\{f_n\}_{n\ge 1}$ be an mth level paperfolding sequence for some $m\ge 1$. Then 
the map $\N\to \Z/2\Z$ defined by $n\to f_n \mod 2$ is a homomorphism from the multiplicative 
semigroup $\N$ to the additive group $\Z/2\Z$. 
\end{lem}
\begin{proof}
By the definition of the paperfolding sequence in Definition \ref{de;paper}, we know $f_{2k}=f_k$. Hence it suffices 
to prove that for odd $k, l\in \N$ we have 
\begin{align}\label{eq;mod 2}
    f_k+f_l\mod 2= f_{kl} \mod 2.
\end{align}
Since $f_k\mod 2= 0$ if and only if $\frac{k-1}{2}$ is even, 
equation (\ref{eq;mod 2}) follows directly from the calculation
\[
\frac{kl-1}{2}= \frac{l(k-1)}{2}+\frac{l-1}{2}.
\]
\end{proof}

We say a sequence $\{\xi_{n}\}_{n\ge 1}$ is periodic eventually if there exist $N,d \in \N$ such that $\xi_{n+d}=\xi_n$ for all
$n\ge N$. Here $d $ is a period of the sequence.

\begin{proof}[Proof of Lemma \ref{transcendence}]
Since $\alpha\in \F_p(\!(t^{-1})\!)$ where  $\F_p$ is the subfield of $\F$ with $p$ elements, it suffices to show that it is not algebraic over $\F_p(t)$.
    We prove this by contradiction. Assume $\alpha$ is algebraic. Then by Christol's Theorem, $\{ f_n\}$ is $p$-automatic with a $4$-tuple $(S,s_0,\mu,\Phi)$
    as in Definition \ref{def;auto}, where
    $\mu= \{\zeta_0,\zeta_1,\ldots,\zeta_{p-1}  \}$. For each positive integer $n$, we let $\xi_n=f_{p^n-1}$. Note that 
    \[
    p^n-1=\sum_{j=0}^{n-1}(p-1)\cdot p^j,
    \]
    hence by (\ref{eq;auto}) we have 
    \[
    \xi_n = \Phi( \zeta_{p-1}^n(s_0)).
    \]
    Since $S$ is a finite set, the sequence $\{\xi_{n}\}_n$ is periodic eventually. 

    Recall that we can write $p=2^mx+1$ for some odd integer $x$ and positive integer $m$. So 
    \begin{align}\label{eq;pn}
            p^{n}-1= (2^mx+1)^{n}-1=  n 2^{m} x+ n (n-1) 2^{2m-1 }x^2  + 2^{3m }y
    \end{align}
    where $y\in\N$. 
    Let $d$ be a periodic of the sequence $\{\xi_{n}\}_n$ for $n\ge N$. We assume without loss of generality that $d$ is even. Then for $n\ge N$ we have 
\begin{align}\label{eq;equiv}
f_{p^{nd}-1}\equiv  f_{p^{3nd}-1} \mod 2.
\end{align}
 In view of (\ref{eq;pn}), we can write $p^{nd}= 1+ 2^l z$ for some odd integer $z$ and $l\ge m+1$. It follows that 
    \[
    p^{3nd}-1= (1+2^l z)^3-1= 3\cdot  2^l z+ 3\cdot  2^{2l }z^2+ 2^{3l}z^3.
    \] In view of Definition \ref{eq;fn}, we have 
    \[
    f_{p^{3nd}-1}= f_{3 z+ 3 \cdot 2^l z^2+ 2^{2l}z^3}= f_{3z}.
    \]
 This together with  Lemma \ref{lem;homo} implies 
 \[
 f_{p^{3nd}-1}\equiv f_3+ f_z\equiv  1+ f_{p^{nd}-1}\mod 2,
 \]
 which contradicts (\ref{eq;equiv}). Therefore, $\alpha$ is transcendental over $\F(t)$. 
\end{proof}

\begin{proof}[Proof of Lemma \ref{lem;trajectory}]

Let  $\alpha$ be as  in Theorem \ref{counterexample}.
Then  $\alpha$ is transcendental over 
$\F(t)$ by Lemma \ref{transcendence}. 
It suffices to show that 
the trajectory 
$A^+_{\infty, \mathfrak p}u(\alpha^{-1}, 0)\Gamma'$ is bounded. 
The semigroup 
$A^+_{\infty, \mathfrak p}$ is the product of a compact group and 
the semigroup \[
S=\{ a{(t^m, t^{-n})}: m\ge   n\ge 0  \}.
\]
Hence it suffices to show that  $Su(\alpha^{-1}, 0)\Gamma'$ is bounded.

Let $\F[\![t]\!] $ be the maximal compact  subring of $\F(t)_{\mathfrak p}$. 
The subgroup $\SL_2(\K\times \F[\![t]\!]) $ of $G'$ induces a homeomorphism 
\[
\SL_2(\K\times \F[\![t]\!])/\SL_2(\F[t])\cong G'/\Gamma'.
\]
The projection map  $\K\times \F[\![t]\!]\to \K$ induces a surjective quotient map 
\[
 \SL_2(\K\times \F[\![t]\!])/\SL_2(\F[t])\to \SL_2(\K)/\SL_2(\F[t])
\]
whose fibers are orbits of the compact group $\SL_2( \F[\![t]\!])$.
So it suffices to show that  the image of $Su(\alpha^{-1}, 0)\Gamma'$ in $\SL_2(\K)/\SL_2(\F[t])$
is bounded. 

Since $a(t^n,t^n)\in \Gamma'$, we have 
\[
a(t^m,t^{-n})u(\alpha^{-1}, 0)\Gamma'=a(t^m,t^{-n})u(\alpha^{-1}, 0)a(t^n,t^n) \Gamma'=a(t^{m+n},1)u (t^{-2n}\alpha^{-1},0 )\Gamma'.
\]
So the image of $Su(\alpha^{-1}, 0)\Gamma'$ is 
\[
\left \{\begin{pmatrix}
    t^{m+n} & 0\\
    0 & t^{-m-n}
\end{pmatrix}\begin{pmatrix}
    1 & t^{-2n} \alpha^{-1}\\
    0 & 1
\end{pmatrix}\SL_2(\F[t]): m\ge n\ge 0 
\right \}.
\]
In view of the function field analogue of Mahler's compactness criterion, see e.g.
\cite{ghosh2007metric}, this set is bounded if and only if 
for all nonzero  $(P, Q)\in \F[t]^2$,
\begin{align}\label{eq;mahler}
\max(| t^{m+n}(Qt^{-2n}\alpha^{-1}+ P)|, |t^{-m-n} Q|)
\end{align}
is bounded away from zero uniformly. 
If  $Q=0$, in view of $|t^{m+n}P|\ge 1$  we have   (\ref{eq;mahler}) is bounded below uniformly by $1$. If $P=0$, then (\ref{eq;mahler}) is bounded below by $|\alpha^{-1}|$ since $m\ge n$. 
 So it suffices to bound (\ref{eq;mahler}) with
$P,  Q\neq 0$.
In this case,  if $|P|\neq |Qt^{-2n}\alpha^{-1}|$, then 
  \[
|Qt^{-2n}\alpha^{-1}+ P|=\max \{|Qt^{-2n}\alpha^{-1}|,  |P| \}\ge 1,
\]
which implies (\ref{eq;mahler}) is bigger than $1$. 
So it suffices to bound (\ref{eq;mahler}) in the case where $P, Q\neq 0$ and 
$|P|= |Qt^{-2n}\alpha^{-1}|$.
In this case it suffices to bound the product of the two terms and show that 
\[
\inf_{0\neq Q , P\in \F[t],n\ge 0}  | P|\cdot |Q+ t^{2n}P\alpha|
>0,
\]
which holds in view of Theorem \ref{counterexample} with $N = P$.
\end{proof}

\section{Embeddings of Homogeneous Spaces}

In this section, we construct an embedding of 
$G'/\Gamma'$ into 
 $\SL_4(\K)/\SL_4(\F[t])$. 
We will construct  an isomorphism  $\eta: G'\to G$ and an embedding 
$\psi: G\to \SL_4(\K )$, both of which descend to embeddings of homogeneous spaces 
(still denoted by $\eta$ and $\psi$ by abuse of notation).
 The relations are expressed in the following diagram.
\begin{equation}\label{fig;one}
\begin{tikzcd}[row sep=3em, column sep=3em]
G' \ar[r, "\eta", "\sim"'] \ar[d, twoheadrightarrow] 
& G \ar[r, hook, "\psi"] \ar[d, twoheadrightarrow] 
& \SL_4(\K) \ar[d, twoheadrightarrow] \\
 G'/\Gamma' \ar[r, "\eta", "\sim"'] 
&  G/\Gamma \ar[r, hook, "\psi"] 
&  \SL_4(\K)/\SL_4(\F[t])
\end{tikzcd}
\end{equation}

Let $\beta\in \K $ satisfies 
 $\beta^2 = 1 + t^{-1}. 
$
Then $L=\F( \beta)$ is a quadratic extension of $\F(t)$.
Note that the field $L$ is the rational function field with variable $\beta$. 
By Hensel's lemma \cite[Theorem 4.6]{Neukirch}, we know that \( L \) is a split at $\infty$, i.e., there are two distinct $\F(t)$-embeddings of $L$ into $\K$. 
We assume without loss of generality that 
\[
\beta = 1 +\langle \beta \rangle .
\]
Then \begin{align}\label{eq;beta}
|\beta+1|=1\quad \mbox{and}\quad|\beta-1|=p^{-1}. 
\end{align}

Note that any element $\theta $ of  \( L \) can be uniquely expressed as \( j + k\beta \), where \( j, k \in \F(t) \). Therefore, \( \theta  \) is integral over \( \F[t] \) if and only if 

\[
2j, j^2+k^2(1+t^{-1}) \in \F[t], 
\]
which holds if and only if 
\( j \in \F[t] \) 
and
\( k \in t\F[t]\). 
Thus, 
$\mathbb{F}[t, t\beta] $
is the integral closure of $\F[t]$ in $L$. 
\begin{align*}
        \frac{\beta + 1}{\beta - 1}& = \frac{(\beta + 1)^2}{\beta^2 - 1} = t(\beta + 1)^2\quad
        \mbox{and}\\
        \frac{\beta - 1}{\beta + 1} &= \frac{(\beta - 1)^2}{\beta^2 - 1} = t(\beta - 1)^2,
\end{align*}
we have 
\begin{align}\label{integral closure}
    \mathbb{F}\left[\frac{\beta + 1}{\beta - 1}, \frac{\beta - 1}{\beta + 1}\right] 
= \mathbb{F}[t(\beta + 1)^2, t(\beta - 1)^2] 
= \mathbb{F}[t(\beta^2 + 1), t\beta] 
= \mathbb{F}[t,t\beta] 
. 
\end{align}

Let  $\tau$ be the non-trivial $\F(t)$-automorphism of $L$  given by $\tau(\beta)=-\beta$. 
It gives a ring embedding 
\[
\phi_L: L\to \K\times\K, \quad \alpha\mapsto (\alpha,\tau(\alpha)).
\]
The group $\Gamma$ is defined to 
be $\SL_2\big (\phi_L(\F[t, t\beta]) \big)$.

The isomorphism $\F(t)\to L$ given by $g(t)\mapsto g\left(\frac{\beta + 1}{\beta - 1}\right)$, where $g$ is a rational function with coefficients in $\F$, induces an 
isomorphism 
\[
\eta: \phi(\F(t))\to \phi_L(L)\quad \mbox{where }
\eta (t, t) = \left(\frac{\beta + 1}{\beta - 1}, \frac{\beta - 1}{\beta + 1} \right).
\]
In view of (\ref{eq;beta}), we have
\begin{align}\label{eq;refer}
     |t^{-1}| = |t|_\mathfrak p = \left| \frac{\beta - 1}{\beta + 1} \right| = p^{-1}, 
\end{align}
which implies that  $\eta$ is continuous. By the weak approximation theorem \cite[Theorem II.3.4]{Neukirch}, both $\phi(\F(t))$ and $\phi_L(L)$ are  dense. So $\eta$ extends to an isomorphism of topological rings 
 \begin{align*}
\eta: \K \times \F(t)_\mathfrak p \to \K \times \K,
\end{align*}
and induces an isomorphism 
$\eta: G' \to G$ of topological groups.
Since 
\begin{align*}
\eta\left(\phi\left(\mathbb{F}[t,t^{-1}]\right) \right)
=\phi_L\left(\mathbb{F}\left[ \frac{\beta + 1}{\beta - 1}, \frac{\beta - 1}{\beta + 1}\right]\right), 
\end{align*}
we have $\eta(\Gamma')=\Gamma$. Therefore
 we have an isomorphism of homogeneous spaces 
$\eta: G'/\Gamma' \to G/\Gamma$. 

Now we construct the map $\psi: G\to \SL_4(\K)$ using the restriction of scalars. 
Let \( \beta_1,\beta_2 \) be   an $\F[t]$-basis  of  $\F[t, t\beta]$. Then $L$ embeds into 
$M_2(\F(t))$ via $\alpha\to l_\alpha $ where
\[
(\alpha\beta_1, \alpha\beta_2)= (\beta_1, \beta_2) l_\alpha. 
\]
The $\F(t)$-linear map $\phi_L(L)\to M_2(\K)$ sending $\phi_L(\alpha)\to l_\alpha$  extends $\K$-linearly  to a homomorphism of topological rings
\[
\psi: \K\times\K\to M_2(\K).
\]
The embedding $\psi$ induces a ring homomorphism 
\[
\psi: M_2(\K\times \K )\to M_2(M_2(\K))\cong M_4(\K),
\]
where we use the same symbol for the induced map  to simplify the notation.
The restriction of $\psi$ to $G$ is the map 
we need. Note that $\Gamma=\psi^{-1}(\SL_4(\F[t]))$ since 
$\F[t, t\beta]$ is the integral closure of $\F[t]$ in $L$.  So  $\psi: G\to \SL_4(\K)$ induces a map from $G/\Gamma$ to $\SL_4(\K)/\SL_4(\F[t])$, which we still denote by $\psi$.

Let  $A$ be the diagonal group of $G$, i.e.
$A=\left\{ a(\alpha_1, \alpha_2) : 0\neq \alpha_1,\alpha_2\in\K   \right\},
$
where we use the notation in (\ref{eq;notation}).
Let 
\[
Z=\left\{\left(\begin{array}{cc}
    (\alpha,\alpha^{-1}) & 0 \\
   0  & (\alpha,\alpha^{-1})
\end{array}\right)   :\quad  0\neq \alpha\in\K   \right\}\subset M_2(\K\times \K).
\]
\begin{lem}\label{embedding}
The map 
$\psi: G/\Gamma\to \SL_4(\K)/\SL_4(\F[t])$ is a proper embedding. Moreover,
there exists $g\in \SL_4(\K)$ such that 
$\psi(AZ)= g^{-1}D g$ where   $D$ is the diagonal group   of  $\SL_4(\K)$ and $\psi(Z) \psi(x)$ is compact for all $x\in G/\Gamma$.  
It follows that 
for all $x\in G/\Gamma$, $A x$ is relatively compact (resp. compact) if and only 
if $Dg\psi(x)$ is relatively compact (resp. compact).
\end{lem}
\begin{proof}
Since $G/\Gamma$ has finite volume, the properness of $\psi$ follows from
\cite[Theorem 1.13]{Raghunathan}.
Since $Z$ commutes with $G$ it suffices to show that $\psi(Z)\psi(x_0)$ is 
compact for $x_0=\Gamma\in G/\Gamma$. 
Let  $\alpha=\frac{\beta-1}{\beta+1}$.  Recall from 
(\ref{integral closure}) and (\ref{eq;refer}) we know that   $\alpha, \alpha^{-1}\in \F[t, t\beta]$ and $|\alpha|<1$. Together with $\det (l_\alpha)= N_{L/\F(t)}(\alpha)=1$, we have 
\[
\gamma=\begin{pmatrix}
    l_\alpha & 0 \\
    0 & l_\alpha
\end{pmatrix}\in \SL_4(\F[t])\cap  \psi(Z).
\]
Therefore, $\psi(Z)\psi(x_0)$ is compact.     

Let  $g_1=\begin{pmatrix}
        \beta_1 & \beta_2\\
        \tau({\beta_1}) &\tau  (\beta_2)
    \end{pmatrix}\in \SL_2(\K)$ and let  $ g=\begin{pmatrix}
        g_1 & 0\\
        0 & cg_1
    \end{pmatrix}\in G$
    where $c=(\beta_1\tau(\beta_2)-\beta_2\tau(\beta_1))^{-1}.$
    It is straight forward to check that  $ g^{-1} D g=\psi(AZ)$. 
\end{proof}

\section{Existence of a Bounded  Non-compact Orbit}
In this section we prove Theorem \ref{thm;main}.  Recall that we have defined maps $\psi$
and $\eta$ satisfying  (\ref{fig;one}).
We will mainly work on $G
/\Gamma$ and push the result to $\SL_4(\K)/\SL_4(\F[t])$ via the map $\psi$. 
Let 
\[
A^+=\left\{a(\alpha_1 , \alpha_2):\alpha_i\in \K, \ |\alpha_1|\ge |\alpha_2|\ge 1  \right\},
\]
which is the image of $A_{\infty, \mathfrak p}^+$ under the isomorphism $\eta$.
The following lemma is an immediate corollary of  Lemma \ref{lem;trajectory}.
\begin{lem}\label{lem;use}
Let $\alpha\in \K$ satisfy   Lemma \ref{lem;trajectory}, let $(\til \alpha, 0)=\eta((\alpha, 0))$ and let
\[
g_0=u(\til \alpha, 0)\in G.
\]
Then $\til \alpha$ is transcendental over $\F(t)$, and 
$A^+g_0\Gamma$
is relatively compact in \(G/\Gamma\).
\end{lem}

For $x\in G/\Gamma$, let $\textup{Stab}_{A}(x) $  denote the elements of $A$ stabilize $x$.
It is a discrete subgroup of $A$. 

\begin{lem}\label{lem;algebraic}
   Let  $x\in G/\Gamma$ be such that 
$Ax$ is compact. Then there exists $h\in G$, whose entries  in $\K$ are algebraic over
$\F(t)$, such that $x= h\Gamma$. Moreover, the closure of 
\[
\{\alpha\in \K: u(\alpha, 0) =au(1, 0)a^{-1}  \mbox{ for some } a\in \textup{Stab}_{A}(x)  \}
\]
in $\K $ is uncountable. 
\end{lem}
\begin{proof}
     Since $A$ is an amenable group, there is an $A$-invariant probability measure supported on \( Ax \). 
   It follows that $A/\textup{Stab}_{A}(x)$ is compact.
   Note that $A=\K^*\times \K^*$ and $\K ^* \cong \Z\times U$, where $U$ is a 
   compact topological group (cf. \cite[Proposition 5.7(ii)]{Neukirch}). Hence we have $\textup{Stab}_{A}(x)\cong \Z^2\times J$, where 
   $J$ is a finite group. 

Let  $\alpha_1,\alpha_2\in \K$ such that 
   $a(\alpha_1, \alpha_2)\in \textup{Stab}_{A}(x).$
We assume that either  $\alpha_1\neq \alpha_1^{-1}$ or $\alpha_2\neq \alpha_2^{-1}$, which holds unless $a(\alpha_1, \alpha_2)$
has finite  order. 
 Write $x=h_1\Gamma$ for some $h_1\in G$. 
  There exists $\gamma\in \Gamma$ such that 
   \[
   a(\alpha_1, \alpha_2)h _1 =h_1 \gamma.
   \]
   Suppose $\gamma=\phi_L(\gamma_0)$, where $\gamma_0\in \SL_2(\F[t, t\beta])$. 
   Note that  $\alpha_1$ is an eigenvalue of $\gamma_0$ and  $\alpha_2$ is an
   eigenvalue of $\tau(\gamma)$, where $\tau$ is the non-trivial Galois conjugate 
   of $L$.
   It follows that $\alpha_i\neq \alpha_i^{-1}$ for $i=1,2$. 
   Therefore,  the  eigenspaces of $\gamma_0$ and $\tau(\gamma_0)$
   are  one dimensional, and there are eigenvectors with entries which are algebraic over $\F(t)$. 
   It follows that there exists $a\in A$ such that $h=ah_1$ has entries 
   algebraic over $\F(t)$.

 The previous argument also implies   that \[
S:=\{\alpha_1^2: a(\alpha_1, {\alpha_2})\in  \textup{Stab}_{A}(x)\} \le \K^*
   \]
   is  a  finitely generated Abelian group with rank $2$. 
   Therefore, $S\cap U$ is an infinite group whose closure is complete and non-discrete. 
   It follows from Baire category theorem  that the closure of $S$ in $\K^*$ is uncountable, which gives
   the second conclusion. 
\end{proof}

The following is an immediate corollary of the  first part of Lemma \ref{lem;algebraic}.
\begin{cor}
    \label{cor;countable}
    There are countably many compact $A$ orbits on $G/\Gamma$.
\end{cor}

\begin{lem}\label{lem;nonper}
 Let  $x\in G/\Gamma$ be such that 
$Ax$ is compact.
Let $g_0$ be as in Lemma \ref{lem;use}.  Then for any $a\in A$ and $$h=
\begin{pmatrix}
    * &  (\theta , *) \\
    * & *
\end{pmatrix}
\in G$$ with 
$ag_0\Gamma= hx$, one has $\theta\neq 0$.
\end{lem}
\begin{proof}
By Lemma \ref{lem;algebraic} there exists
  $h_0\in \Gamma$ whose entries in $\K$ are all algebraic over 
 $\F(t)$ such that 
 $x=h_0\Gamma$.
 Since $ag_0\Gamma= hh_0\Gamma$, there exists $\gamma\in \Gamma$ such that 
$$ag_0 = hh_0\gamma= \begin{pmatrix}
    (\theta_1, * ) & (\theta_2, *)\\
    * & *
\end{pmatrix} .$$
Therefore, $\frac{\theta_2}{\theta_1}=\til \alpha$ is transcendental over $\F(t)$.
In particular, $\theta \neq 0$. 
\end{proof}

Let $\pi_i: A\to \K^*$ be the map defined by $\pi_i(a(\alpha_1, \alpha_2))=\alpha_i$ where $a(\alpha_1, \alpha_2)$ is defined as in (\ref{eq;notation}).
We say  a sequence $\{ a_n \}\subset A^+$ is  divergent if
\[
\lim_{n\to \infty} |\pi_1 (a_n)\pi_2(a_n)^{-1}|= \lim_{n\to \infty}|{\pi_2(a_n)}| =\infty.
\]
Let \[
\mathcal L=\{x\in G/\Gamma: x= \lim_{n\to \infty} a_n g_0\Gamma 
\mbox{ for some  divergent  sequence } \{a_n \}\subset A^+ \}.
\]
Since $A_0^+g_0\Gamma$ is relatively compact, the set  $\mathcal L$ is nonempty and compact. 

\begin{lem}\label{bounded}
    For any $x\in \mathcal L$ we have  \( Ax \subset \mathcal L\). Hence $Ax$ is relatively compact in \( G/\Gamma \).
\end{lem}

\begin{proof}
    Suppose $a\in A$ and  $x=\lim_{n} a_n g_0 \Gamma$ for some divergent sequence $\{a_n \}$ of $A^+$. 
   By the definition of the divergent sequence, we have $a_n a\in A^+$ for 
   $n$ sufficiently large.
Therefore
\begin{align*}
        ax
        = a \lim_n a_{n}g_0\Gamma 
        = \lim_n (aa_n)g_0\Gamma \in \mathcal L.
    \end{align*}
\end{proof}

\begin{thm}\label{eq;thm0}
    There exists \(y\in \mathcal L \) such that \( A y \) is a non-closed relatively compact subset of $G/\Gamma$.
\end{thm}

\begin{proof}
    We fix $x\in \mathcal L$. By Lemma \ref{bounded}, 
      \( Ax \) is relatively compact. If \( Ax \) is non-closed, we are done. Now we  assume it is closed and hence compact. 
    It follows that $A/\textup{Stab}_{A}(x)$ is compact. 
    Let $F$ be a compact subset of $A$ such that $A=F\cdot\textup{Stab}_{A}(x)$.

    Let $\{a_n\}$ be a divergent sequence of  $A$ such that 
    $ \lim_{n\to \infty} a_ng_0\Gamma = x.$
    Then there is a sequence $\{ h_n\}\subset G$ such that \[
    a_n g_0\Gamma= h_n x \quad \mbox{and}\quad h_n\to I
    \] 
    where  $I$ is the identity element of $G$.
  Suppose   $  h_n=\begin{pmatrix}
     * & (\theta_{n}, *)\\
    * &    *
    \end{pmatrix}$, 
  then  Lemma \ref{lem;nonper} implies   that $\theta_n\neq 0$ for all $n$. 
So there exists a sequence $b_n=a(\alpha_n, 1)\in A^+$ such that 
\[
b_nh_n b_n^{-1} \to h=u(\theta', 0)
\]
for some nonzero $\theta'\in \K $. 
Since $A=F\cdot \textup{Stab}_{A}(x)$ we can write 
$b_n= c_nf_n$ where $c_n\in \textup{Stab}_{A}(x)$ and $f_n\in F$. 
We assume without loss of generality that $f_n\to f\in F$
  
    Thus, there exists a nonzero  $\theta \in \K$ such that
    \[
    c_na_ng_0 \Gamma= c_n h_n x= c_n h_n c_n^{-1} x\to   u(\theta, 0)x. 
    \]
    Note that $a_nc_n= b_nf_n^{-1}a_n\in A^+$ for $n$ sufficiently large and 
    is divergent. 
    It follows that $u(\theta, 0)x\in \mathcal L$.
    By Lemma \ref{bounded}, we have \( A u(\theta,0) x\subseteq \mathcal L \)  is  bounded.
    
By Lemma \ref{lem;algebraic},     \[
  S:=  \overline{\{ b u(\theta, 0) b^{-1} : b \in \mathrm{Stab}_{A} x \}}
    \] 
    is  uncountable. Note that  $hx\in \mathcal L$ for all $h\in S$.  
But  there are only countably many compact $A$ orbits in $G/\Gamma$ by Corollary \ref{cor;countable}. Therefore, 
    there exists $h\in S$ such that  $ y= hx\in \mathcal L $ such that $Ay $ is non-compact. 
    By Lemma \ref{bounded}, $A y\subset \mathcal L$ and hence is bounded. 
\end{proof}

\begin{proof}[Proof of Theorem \ref{thm;main}]
    By Theorem \ref{eq;thm0} there exists $y\in G/\Gamma$ such that 
    $A y$ is bounded and non-closed. 
 Let $g\in G$ be as in Lemma \ref{embedding} such that 
 $g^{-1} D g=\psi(A Z)$. Then, $Ag \psi(y)$ is a bounded and non-closed subset 
 of $\SL_4(\K)/\SL_4(\F[t])$ by Lemma \ref{embedding}. 
\end{proof}

\bibliographystyle{plain}
\bibliography{refs}

\begin{thebibliography}{10}

\bibitem{Erez}
Faustin Adiceam, Erez Nesharim, and Fred Lunnon.
\newblock On the {$t$}-adic {L}ittlewood conjecture.
\newblock {\em Duke Math. J.}, 170(10):2371--2419, 2021.

\bibitem{2003Automatic}
Jean~Paul Allouche and Jeffrey~O. Shallit.
\newblock {\em Automatic Sequences. Theory, Applications, Generalizations}.
\newblock Automatic Sequences. Theory, Applications, Generalizations, 2003.

\bibitem{Cassels55}
J.~W.~S. Cassels and H.~P.~F. Swinnerton-Dyer.
\newblock On the product of three homogeneous linear forms and the indefinite
  ternary quadratic forms.
\newblock {\em Philos. Trans. Roy. Soc. London Ser. A}, 248:73--96, 1955.

\bibitem{EKL}
Manfred Einsiedler, Anatole Katok, and Elon Lindenstrauss.
\newblock Invariant measures and the set of exceptions to {L}ittlewood's
  conjecture.
\newblock {\em Ann. of Math. (2)}, 164(2):513--560, 2006.

\bibitem{ELM}
Manfred Einsiedler, Elon Lindenstrauss, and Amir Mohammadi.
\newblock Diagonal actions in positive characteristic.
\newblock {\em Duke Math. J.}, 169(1):117--175, 2020.

\bibitem{garrett2025counterexamplesptadiclittlewoodconjecture}
Samuel Garrett and Steven Robertson.
\newblock Counterexamples to the $p(t)$-adic littlewood conjecture over small
  finite fields, 2025.

\bibitem{ghosh2007metric}
Anish Ghosh.
\newblock Metric diophantine approximation over a local field of positive
  characteristic.
\newblock {\em Journal of Number Theory}, 124(2):454--469, 2007.

\bibitem{margulis2000}
Gregory Margulis.
\newblock Problems and conjectures in rigidity theory.
\newblock In {\em Mathematics: frontiers and perspectives}, pages 161--174.
  Amer. Math. Soc., Providence, RI, 2000.

\bibitem{Neukirch}
J\"urgen Neukirch.
\newblock {\em Algebraic number theory}, volume 322 of {\em Grundlehren der
  mathematischen Wissenschaften [Fundamental Principles of Mathematical
  Sciences]}.
\newblock Springer-Verlag, Berlin, 1999.
\newblock Translated from the 1992 German original and with a note by Norbert
  Schappacher, With a foreword by G. Harder.

\bibitem{Raghunathan}
M.~S. Raghunathan.
\newblock {\em Discrete subgroups of {L}ie groups}, volume Band 68 of {\em
  Ergebnisse der Mathematik und ihrer Grenzgebiete [Results in Mathematics and
  Related Areas]}.
\newblock Springer-Verlag, New York-Heidelberg, 1972.

\end{thebibliography}

\end{document}